\documentclass[11pt, letterpaper,reqno]{amsart}

\usepackage{amsmath,amsthm,amsrefs,amssymb,hyperref,caption,color}
\usepackage[percent]{overpic}

\numberwithin{equation}{section}
\newtheorem{question}{Question}[section]

\newtheorem{corollary}{Corollary}[section]
\newtheorem{lemma}{Lemma}[section]
\newtheorem{theorem}{Theorem}[section]

\theoremstyle{definition}

\DeclareMathOperator{\D}{\mathbb{D}}
\DeclareMathOperator{\C}{\mathbb{C}}

\begin{document}
	\title[Extremal distance and mappings in Hardy and Bergman spaces]{Extremal distance and conformal mappings in Hardy and Bergman spaces}
	
	\author{Christina Karafyllia}  
	 \address{Department of Mathematics, University of Thessaly, Lamia, 35100, Greece}
	 \email{ckarafyllia@uth.gr}
	 	\address{Institute for Mathematical Sciences, Stony Brook University, Stony Brook, NY 11794, U.S.A.}
	\email{christina.karafyllia@stonybrook.edu}

	\subjclass[2010]{Primary 30H10, 30H20, 31A15; Secondary 42B30, 30C85}
	
	\keywords{Extremal distance, reduced extremal distance, harmonic measure, Hardy spaces, weighted Bergman spaces, Hardy number}
	
	\begin{abstract} We prove necessary and sufficient integral conditions involving extremal distance for a conformal mapping of the unit disk to belong to the Hardy or weighted Bergman spaces. We also give characterizations for the Hardy number and the Bergman number of a simply connected domain in terms of extremal distance.
	\end{abstract}
	
	\maketitle
	
	\section{Introduction}\label{int}

In this paper we study the problem of finding necessary and sufficient geometric conditions for a conformal mapping of the unit disk $\D$ to belong to some Hardy or weighted Bergman space. Such conditions have been proved in terms of harmonic measure and hyperbolic distance. We establish new conditions involving extremal distance.
  
The Hardy space with exponent $p>0$ is denoted by $H^p (\D)$ and is defined to be  the set of all holomorphic functions $f$ of $\D$ such that
\[ \sup_{0<r<1}\int_{0}^{2\pi} {{{| {f( re^{i\theta} )}|}^p} d\theta}  <  +\infty.\]
For the theory of Hardy spaces, see \cite{Dur}. Let $D\ne \mathbb{C}$ be a simply connected domain and $f$ be a Riemann mapping from $\mathbb{D}$ onto $D$. The Hardy number of $D$, or equivalently of $f$, is defined by 
\[{h}\left( D \right)=h(f) = \sup \left\{ {p > 0:f \in {H^p}( \mathbb{D})} \right\}.\]
This definition is independent of the choice of the Riemann mapping onto $D$ \cites{Han,Kim}. Since every conformal mapping of $\mathbb{D}$ belongs to ${H^p}\left( \mathbb{D} \right)$ for all $p \in (0,1/2)$ \cite[p.\ 50]{Dur}, the number  ${h}(D)$ lies in $[1/2, +\infty]$.

A more general class of holomorphic functions that contains Hardy spaces is the class of weighted Bergman spaces. The weighted Bergman space with exponent $p>0$ and weight $\alpha>-1$ is denoted by $A_\alpha ^p (\D)$ and is defined to be the set of all holomorphic functions $f$ of $\mathbb{D}$ such that
\[ \int_\mathbb{D} {{{\left| {f\left( z \right)} \right|}^p}{{( {1 - {{| z |}^2}})}^\alpha }dA\left( z \right)}  < +\infty,
\] 
where $dA$ denotes the Lebesgue area measure on $\mathbb{D}$. For the theory of Bergman spaces, see \cite{DurS}. As an  analogue of the Hardy number for weighted Bergman spaces, the current author and Karamanlis introduced in \cite{Karjournal} the Bergman number as follows. If $f$ is a conformal mapping of $\D$, the Bergman number of $f$ is defined by
\[b(f)=\sup\left\{ \frac{p}{\alpha+2}:\ f\in A_{\alpha}^p(\D),p>0,\alpha>-1 \right\}.\]
They then proved that $b(f)=h(f)$. See \cites{Karjournal,Kabergman}.

A well-studied problem in geometric function theory is to find geometric conditions for a conformal mapping $f$ of $\D$ to belong to some space ${H^p}( \mathbb{D})$ or $A_\alpha ^p (\D)$ by studying the image region $f(\D)$. See, for example, \cites{Karjournal,Karams,Cor} and references therein. Some of these conditions that have been proved recently involve conformal invariants such as harmonic measure and hyperbolic distance. Before we state the results, we fix some notation. 

For a domain $\Omega$ in $\C$, a point $z \in \Omega$ and a Borel subset $A$ of $\overline \Omega$, let ${\omega _{\Omega}}\left( {z,A} \right)$ denote the harmonic measure of $A$ at $z$ with respect to the component of $\Omega \backslash A$ containing $z$. Also, we denote by $d_\Omega$ the hyperbolic distance in $\Omega$. For the general theory of harmonic measure and hyperbolic distance, see \cite{Gar}.

Henceforth, let $f$ be a conformal mapping of $\D$. For $r>0$, we set $F_r=\{z\in\D:\ |f(z)|=r\}$. 
In \cites{Cor,Pog} Poggi-Corradini gave a necessary and sufficient integral condition for $f$ to belong to $H^p (\D)$ in terms of the harmonic measure $\omega_{\D} (0,F_r)$. The current author and Karamanlis extended this condition to weighted Bergman spaces in \cite{Karjournal}. In \cite{Karams} the current author established another necessary and sufficient integral condition involving this time the  hyperbolic distance $d_{\D}(0,F_r)$. In \cite{Karcanadian} Betsakos, the author and Karamanlis  generalized this condition to weighted Bergman spaces. We state all these results in the following theorem. For simplicity, we use the convention that $A^{p}_{-1}(\D)=H^p(\D)$ for every $p>0$.

\begin{theorem}[\cites{Cor,Pog,Karjournal,Karams,Karcanadian}]\label{known}
Let $f$ be a conformal mapping of $\D$. If $p>0$ and $\alpha \ge -1$, the following statements are equivalent.
\begin{enumerate}
\item $f\in A_{\alpha}^p (\D)$, \smallskip
\item $\displaystyle \int_{0}^{+\infty} r^{p-1}\omega_{\D} (0,F_r)^{\alpha+2}dr<+\infty$, \smallskip
\item $\displaystyle \int_{0}^{+\infty} r^{p-1}e^{-(\alpha+2)d_{\D}\left(0, F_r\right)}dr<+\infty$.
\end{enumerate}
\end{theorem} 

Harmonic measure and hyperbolic distance can also give characterizations for the Hardy and the Bergman number of a conformal mapping of $\D$. These results are due to Kim and Sugawa \cite{Kim} and the current author \cite{Kararkiv}, respectively.

\begin{theorem}[\cites{Kararkiv,Kim}]\label{haknown} Let $f$ be a conformal mapping of $\D$. Then
\[h(f) =b(f)= \mathop {\liminf}\limits_{r  \to  + \infty } \frac{{\log {\omega _\mathbb{D}}{{\left( {0,{F_r}} \right)}^{ - 1}}}}{{\log r }}=\mathop {\liminf}\limits_{r  \to  + \infty } \frac{{{d_\mathbb{D}}\left( {0,{F_r }} \right)}}{{\log r }}.\]
\end{theorem}

We now discuss new results that give similar integral conditions and characterizations in terms of extremal distance and reduced extremal distance. We observe that the equivalence of (2) and (3) in Theorem \ref{known} is not trivial since, in general, $\omega_{\D} (0,F_r)$ and  $e^{-d_{\D}(0, F_r)}$ are not comparable. More precisely, the Beurling-Nevanlinna projection theorem \cite[p.\ 9]{Pog} implies that, for every $r>0$,
\[ \omega_{\D} (0,F_r)\ge \frac{2}{\pi} e^{-d_{\D} (0,F_r)}\]
but the reverse inequality fails in general as the author proved in \cite{Karindiana}. However, integrating as in Theorem \ref{known} implies that the integrals are finite or infinite at the same time. Furthermore, taking limits as in Theorem \ref{haknown} shows that  $\omega_{\D} (0,F_r)$ and  $d_{\D}(0, F_r)$ can be equally used to determine the Hardy number.

In the same spirit, a possible question is whether we could obtain similar conditions involving the extremal distance $\lambda$ or the reduced extremal distance $\delta$. Actually, it is known that \cite[p.\ 164]{Gar}, for every $r>0$,
\[\omega_{\D} (0,F_r)\le e^{-\pi \delta_{\D} (0,F_r)}. \] 
The reverse inequality fails in general if $F_r$ has more than one component. However, even though $\omega_{\D} (0,F_r)$ and  $e^{-\delta_{\D}(0, F_r)}$ are not comparable, one might expect that we could obtain similar conditions to  Theorem \ref{known} by integrating. In other words, the following question arises.
\begin{question}\label{que} Let $p>0$ and $\alpha\ge -1$. Is it true that $f\in A_{\alpha}^p(\D)$ if and only if 
\[\int_{0}^{+\infty} r^{p-1} e^{-\pi (\alpha+2) \delta_{\D} (0,F_r)} dr<+\infty?\] 
\end{question}
This question was posed to me by Pietro Poggi-Corradini in personal communication. In Section \ref{exa} we prove that the answer is negative by providing a counter-example. Moreover, we found how to modify the integral condition so as to get a positive answer.

First, we observe that if $F_r$ had only one component then $\omega_{\D} (0,F_r)$  and $e^{-\pi \delta_{\D} (0,F_r)}$ would be comparable \cite[p.\ 164]{Gar}. Consequently, the key idea is to consider one component of $F_r$. Studying the problem in detail, we found that the appropriate component is the following one. Let $I_r$ be an enumeration of the components of $F_r$, which we denote by ${\left\{ {F_r^i } \right\}_{i \in I_r } }$. We consider a component $F_r^*$ such that
\[\omega_{\D}(0,F_r^*)=\max \{\omega_{\D}(0,F_r^i):i \in I_r\}.\]
We note that the existence of $F_r^*$ is not trivial since the harmonic measures $\omega_{\D}(0,F_r^i)$, $i\in I_r$, are considered in different domains. However, in Section \ref{pre} we show that there is always such a component which also satisfies 
\[\delta_{\D} (0,F_r^*)=\min \{ \delta_{\D} (0,F_r^i): i\in I_r\}\,\,\,{\rm and}\,\,\,\lambda_{\D} (0,F_r^*)=\min \{ \lambda_{\D} (0,F_r^i): i\in I_r\}.\]
Therefore, in Section \ref{pro} we prove the following result.

\begin{theorem}\label{main}
Let $f$ be a conformal mapping of $\D$. If $p>0$ and $\alpha \ge -1$, the following statements are equivalent.
\begin{enumerate} 
	\item $f\in A_{\alpha}^p (\D)$,	
	\smallskip
	\item $\displaystyle \int_{0}^{+\infty} r^{p-1} \omega_{\D} (0,F_r^*)^{\alpha+2} dr<+\infty$, \smallskip
	\item $\displaystyle \int_{0}^{+\infty} r^{p-1}e^{-\pi (\alpha+2) \delta_{\D} (0,F_r^*)} dr<+\infty$, \smallskip
	\item $\displaystyle \int_{0}^{+\infty} r^{p-1}e^{-\pi (\alpha+2) \lambda_{\D} (0,F_r^*)} dr<+\infty$.
\end{enumerate}
\end{theorem}

Using $\omega_{\D} (0,F_r^*),\delta_{\D} (0,F_r^*)$ and $\lambda_{\D} (0,F_r^*)$ we can also give characterizations for the Hardy and the Bergman number of $f$ as follows. 

\begin{theorem}\label{hardynumber} Let $f$ be a conformal mapping of $\D$. Then 
\begin{enumerate}
	\item $\displaystyle h(f)=b(f)=\liminf_{r\to +\infty}\frac{\log \omega_{\D} (0,F_r^*)^{-1} }{\log r}$, \smallskip
	\item $\displaystyle h(f)=b(f)=\liminf_{r\to +\infty}\frac{\pi\delta_{\D} (0,F_r^*)}{\log r}$,\smallskip
	\item $\displaystyle h(f)=b(f)=\liminf_{r\to +\infty}\frac{\pi\lambda_{\D} (0,F_r^*)}{\log r}$.
\end{enumerate}

\end{theorem}

\section{preliminaries}\label{pre}

\subsection{Extremal distance}

Let $D$ be a plane domain and $E,F$ be two disjoint sets on $\partial D$. Let $\Gamma$ be the family of all rectifiable curves in $D$ joining $E$ to $F$. We consider non-negative Borel measurable functions $\rho$ in $D$ and define 
\[L\left( {\Gamma ,\rho } \right) = \mathop {\inf }\limits_{\gamma  \in \Gamma } \int_\gamma  {\rho \left| {dz} \right|}\,\,\,{\rm and}\,\,\,A\left( {D,\rho } \right) = \int \int_D {{\rho ^2}dxdy}.\]
The extremal distance $\tilde{\lambda}_D (E,F)$ between $E$ and $F$ in $D$ is defined by
\[ \tilde{\lambda}_D (E,F)= \mathop {\sup }\limits_\rho  \frac{{L{{\left( {\Gamma ,\rho } \right)}^2}}}{{A\left( {D,\rho } \right)}},\]
where the supremum is taken over all $\rho$ that satisfy $0 < A\left( {D,\rho } \right) <  + \infty $. Now, let $D$ be a Jordan domain in $\C$, $E$ be an arc on $\partial D$ and $z_0\in D$. Consider all Jordan arcs $\sigma \subset D$ joining $z_0$ to $\partial D \backslash E$ and define 
\[\lambda_D(z_0,E)=\sup_{\sigma}\tilde{\lambda}_{D\backslash \sigma} (\sigma,E),\]
where the supremum is taken over all such Jordan arcs. The quantity $\lambda_D(z_0,E)$ is conformally invariant \cites{Beu,Gar} and it is related to the harmonic measure $\omega_D (z_0,E)$ \cite[p.\ 145]{Gar} in the following way
\begin{equation}\label{haex}
e^{-\pi\lambda_D(z_0,E)}\le\omega_D (z_0,E)\le \frac{8}{\pi}e^{-\pi\lambda_D(z_0,E)}.
\end{equation}

\subsection{Reduced extremal distance}

Let $D$ be a finitely connected Jordan domain in $\C$, let $E$ be a finite union of subarcs of $\partial D$ and $z_0\in D$. For $\varepsilon >0$, we set $B_{\varepsilon}=B(z_0,\varepsilon)=\{z\in \C:|z-z_0|<\varepsilon\}$. The reduced extremal distance $\delta_D(z_0,E)$ \cite{Gar} is defined by
\[\delta_D(z_0,E)=\lim_{\varepsilon \to 0}(\tilde{\lambda}_{D\backslash B_{\varepsilon}}(\partial B_{\varepsilon},E)-\tilde{\lambda}_{D\backslash B_{\varepsilon}}(\partial B_{\varepsilon},\partial D)).\]
The reduced extremal distance is conformally invariant \cite[p.\ 163]{Gar}. If $E$ is a single arc on $\partial D$, then \cite[p.\ 164]{Gar} $\delta_D(z_0,E)$ is related to the harmonic measure $\omega_D (z_0,E)$ in the following way
\begin{equation}\label{hare}
\frac{2}{\pi}e^{-\pi \delta_D(z_0,E)}\le \omega_D (z_0,E)\le e^{-\pi \delta_{D}(z_0,E)}.
\end{equation}

\subsection{Harmonic measure} Now, we state two properties of harmonic measure that we use in the proofs below. The first one is a corollary of the Beurling-Nevanlinna projection theorem. 

\begin{lemma}[{\cite[p.\ 107]{Gar}}]\label{harbeu}
Let $D$ be an unbounded simply connected domain, $z \in D$ and $\zeta \in \partial D$. Then, for $0<r<|z-\zeta|$,
\[ \omega_D (z, B(\zeta,r)\cap \partial D)\le K\left( \frac{r}{|z-\zeta|}\right)^{1/2},\]
where $K>0$ is a constant.
\end{lemma}

The second result we need is an immediate consequence of the strong Markov property for harmonic measure \cite[p.\ 282]{Bet}.

\begin{lemma}\label{markov}
Let $D_1,D_2\subset \C$ be two simply connected domains. Assume that $D_1 \subset D_2$ and let $F\subset \partial D_2 \backslash \partial D_1$ and $\sigma =\partial D_1 \backslash \partial D_2$. Then, for $z\in D_1$,
\[\omega_{D_2} (z,F) \le \omega_{D_1} (z, \sigma).\]
\end{lemma}

\subsection{Auxiliary lemmas} From now on, let $f$ be a conformal mapping from $\D$ onto an unbounded simply connected domain. For $r>0$, set $F_r=\{z\in\D:\ |f(z)|=r\}$. Note that $f(F_r)=f(\D)\cap\{z\in \C: |z|=r\}$ is the union of countably many open arcs in $f(\D)$. By Proposition 2.14 in \cite[p.\ 29]{Pom} it follows that $F_r$ is the union of countably many analytic open arcs in $\D$ so that each such arc has two distinct endpoints on $\partial\D$.  Let $N(r)\in \mathbb{N}\cup \{+\infty\}$ be the number of components of $F_r$. Then we set
\[I_r = \left\{ \begin{array}{l}
\left\{ {1,2, \ldots ,N\left( r \right)} \right\},\,{\rm{if}}\,N\left( r \right) <  + \infty  \\ 
\mathbb{N},\,\,\,\,\,\,\,\,\,\,\,\,\,\,\,\,\,\,\,\,\,\,\,\,\,\,\,\,\,\,\,\,\,\,\,\,\,\,\,\,{\rm{if}}\,N\left( r \right) =  + \infty  \\ 
\end{array} \right.\] 
and denote by ${\left\{ {F_r^i } \right\}_{i \in I_r } }$ the components of $F_r$.

\begin{lemma}\label{maxi} For every $r>0$, there exists a component $F_r^*$ of $F_r$ such that
\[\omega_{\D}(0,F_r^*)=\max \{\omega_{\D}(0,F_r^i):i \in I_r\}.\]
\end{lemma}
\begin{proof}
Let $f(\D)=D$. Fix an $r>|f(0)|$ and let $C=\{z\in \C: |z|=r\}$. Then, since $D$ is unbounded, $D\cap C \neq \emptyset$. We now apply Proposition 2.13 in \cite[p.\ 28]{Pom}. Since $f(0)\in D\backslash C$, there are countably many crosscuts $C_k \subset C$, $k\in J$, of $D$ such that
\[D=D_0 \cup \bigcup_{k\in J}D_k\cup \bigcup_{k\in J} C_k,\]
where $D_0$ is the component of $D\backslash C$ containing $f(0)$ and $D_k$ are disjoint domains with
\[C_k=D\cap \partial D_k \subset D\cap \partial D_0\]
for $k\in J$. It follows that if $z\in D$ lies in the unbounded component of $\C \backslash C$ then $z\in D_k$ for exactly one $k$. Note that $D_k$, for $k\in J$, may contain points of $C$. See for example $D_1$ in Fig. \ref{pomme}.

We notice that since $D$ is unbounded, there is at most one $C_{k}$ such that the component of $D\backslash C_{k}$ containing $f(0)$ is bounded. Indeed, since $D$ is unbounded, there is some $k^*\in J$ such that the component $D_{k^*}$ of $D\backslash C_{k^*}$ not containing $f(0)$ is unbounded (see for example $D_2$ in Fig. \ref{pomme}). Then the component of $D\backslash C_{k^*}$ containing $f(0)$ is either bounded or unbounded. For every $k\in J\backslash \{k^*\}$, the component of $D\backslash C_{k}$  containing $f(0)$ is unbounded since it contains $D_{k^*}$.

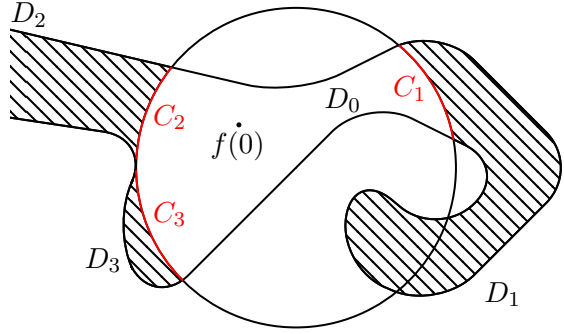
\begin{figure}
\begin{overpic}[width=0.6\textwidth]{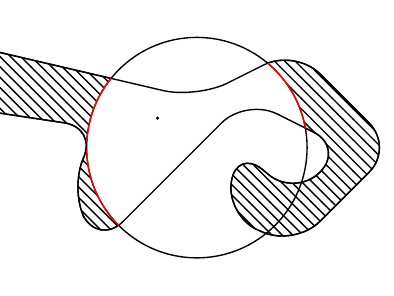}
\put  (35,40) {$f(0)$}
\put  (55,48) {$D_0$}
\put  (83,14) {$D_1$}
\put  (0,63) {$D_2$}
\put  (13,20) {$D_3$}
\put  (67,50) {\color{red} $C_1$}
\put  (25,45) {\color{red}$C_2$}
\put  (25,28) {\color{red}$C_3$}
\end{overpic}
\caption{The domains $D_k$ and the arcs $C_k$.}
\label{pomme}
\end{figure}

Now, let $B_i$ denote the component of $D\backslash f(F_r^i)$ containing $f(0)$. Then by the conformal invariance of harmonic measure we have
\begin{align}
\sup \{\omega_{\D }(0,F_r^i):i \in I_r\}=\sup \{\omega_{B_i}(f(0),f(F_r^i)):i \in I_r\}. \nonumber
\end{align}
We claim that
\[\sup \{\omega_{B_i}(f(0),f(F_r^i)):i \in I_r\}=\sup \{\omega_{B_k}(f(0),C_k):k \in J\}.\]
Indeed, we observe that for each $i\in I_r$ we have either $f(F_r^i)=C_k$ for some $k$ or, by Lemma \ref{markov},
\[\omega_{B_i}(f(0),f(F_r^i))\le \omega_{B_k}(f(0),C_k)\]
for some $k$. This implies that
\[\sup \{\omega_{B_i}(f(0),f(F_r^i)):i \in I_r\}\le\sup \{\omega_{B_k}(f(0),C_k):k \in J\}. \]
The reverse inequality comes directly from the fact that
\[\{\omega_{B_k}(f(0),C_k):k \in J\}\subset \{\omega_{B_i}(f(0),f(F_r^i)):i \in I_r\}.\]
Therefore, 
\begin{equation}\label{sup}
\sup \{\omega_{\D }(0,F_r^i):i \in I_r\}=\sup \{\omega_{B_k}(f(0),C_k):k \in J\}.
\end{equation}

 If $J$ is finite then the proof is complete. So, we suppose that $J$ is infinite. As we proved above, since $D$ is unbounded, there is at most one $C_{k}$ such that the component of $D\backslash C_{k}$ containing $f(0)$ is bounded. For every $k\in J$, let $l(C_k)$ be the length of the circular arc $C_k$ and $\zeta_k$ be the midpoint of the arc $C_k$. Since the series 
\[0\le \sum\limits_{k\in J }  l(C_k) \le 2\pi r\]
converges, we infer that
\begin{equation}\label{limit}
\lim_{k \to  + \infty } l(C_k) = 0
\end{equation}
So, there is a $k_1 \in \mathbb{N}$ such that $l(C_k)<r-|f(0)|$ and $B_k$ is unbounded for every $k\ge k_1$. By Lemma \ref{harbeu} it follows that, for $k \ge k_1$ we have
\begin{equation}\label{anis}
\omega_{B_k}(f(0), B(\zeta_k,l(C_k))\cap \partial B_k)\le K\left( \frac{l(C_k)}{r-|f(0)|} \right)^{1/2},
\end{equation}
where $K>0$ is a constant. Since $C_k \subset B(\zeta_k,l(C_k))\cap \partial B_k$, by monotonicity, we get
\[\omega_{B_k} (f(0), C_k)\le\omega_{B_k}(f(0), B(\zeta_k,l(C_k))\cap \partial B_k).\]
This in combination with (\ref{anis}) and (\ref{limit}) gives
\[\lim_{k \to  + \infty } \omega_{B_k} (f(0), C_k)= 0.\]
Therefore, there is a $ {k_2} > k_1$ such that $\omega_{B_k} (f(0), C_k) \le \omega_{B_1} (f(0), C_1)$ for every $k \ge {k_2}$. This implies that there exists an arc $C_{k^*}$, where $k^* \in\{1,2,\dots,k_2\}$, such that
\[\omega_{B_{k^*}}(f(0),C_{k^*})=\max \{\omega_{B_k}(f(0),C_k):k \in J\}.\]
By this and (\ref{sup}) it follows that there is a component $F_r^*$ of $F_r$ such that
\[\omega_{\D}(0,F_r^*)=\max \{\omega_{\D}(0,F_r^i):i \in I_r\}\]
and the proof is complete.
\end{proof}

\begin{lemma}\label{decreasing} The functions $\omega_{\D}(0,F_r^*),e^{-\pi \delta_{\D}(0,F_r^*) }$ and $e^{-\pi \lambda_{\D}(0,F_r^*) }$ are decreasing in $r>0$ and thus they are measurable. Moreover,
\[\delta_{\D} (0,F_r^*)=\min \{ \delta_{\D} (0,F_r^i): i\in I_r\}\,\,\,and\,\,\,\lambda_{\D} (0,F_r^*)=\min \{ \lambda_{\D} (0,F_r^i): i\in I_r\}.\]
\end{lemma}

Note that the extremal distances $\delta_{\D}(0,F_r^i)$ and $\lambda_{\D}(0,F_r^i)$ are considered with respect to the component of $\D \backslash F_r^i$ containing $0$. Here, a function $f\colon I\to \mathbb R$ on an interval $I\subset \mathbb R$ is decreasing if $r_1<r_2$ implies $f(r_1)\ge f(r_2)$.

\begin{proof} Fix an $r>0$. For every $i\in I_r$, by the Riemann mapping theorem there is a conformal mapping $g_i$ from the component of $\D\backslash F_r^i$ containing $0$ onto $\D$ so that $g_i(0)=0$ and $g_i(F_r^i)$ is the arc $(e^{-i\theta_i},e^{i\theta_i})$, where $\theta_i \in (0, \pi ]$. By the conformal invariance of harmonic measure, we infer that, for every $i\in I_r$,
\begin{equation}\label{harmonicde}
\omega_{\D} (0,F_r^i)=\omega_{\D}(0,g_i(F_r^i))=\frac{\theta_i}{\pi}. 
\end{equation}
Moreover, the conformal invariance of extremal distance implies that 
\[\lambda_{\D} (0,F_r^i)=\lambda_{\D} (0,g_i(F_r^i))=\tilde{\lambda}_{\D} ([-1,0],g_i(F_r^i)).\]
For the last equality, see \cite[p.\ 370]{Beu}. When $\theta_i$ increases, the monotonicity of extremal length  implies that $\lambda_{\D} (0,F_r^i)$ decreases. This in combination with (\ref{harmonicde}) and Lemma \ref{maxi} gives
\[\lambda_{\D} (0,F_r^*)=\min \{ \lambda_{\D} (0,F_r^i): i\in I_r\}.\]
Similarly,
\[\delta_{\D} (0,F_r^*)=\min \{ \delta_{\D} (0,F_r^i): i\in I_r\}.\]

Now, let $r_1<r_2$. By Lemma \ref{markov} and the definition of the component $F_r^*$ we have, respectively, that
\[\omega_{\D}(0,F_{r_2}^*)\le \omega_{\D}(0,F_{r_1}^k) \le \omega_{\D}(0,F_{r_1}^*), \]
where $F_{r_1}^k$ is a component of $F_{r_1}$ that separates $0$ from $F_{r_2}^*$. So, $\omega_{\D}(0,F_r^*)$ is a decreasing function of $r$.

For $i=1,2$, by the Riemann mapping theorem there is a conformal mapping $h_i$ from the component of $\D\backslash F_{r_i}^*$ containing $0$ onto $\D$ so that $h_i(0)=0$ and $h_i(F_{r_i}^*)$ is the arc $(e^{-i\phi_i},e^{i\phi_i})$, where $\phi_i \in (0, \pi ]$. By the conformal invariance and the monotonicity of the harmonic measure $\omega_{\D} (0,F_r^*)$ in $r$, we deduce that $\phi_1 \ge \phi_2$. This in conjunction with the conformal invariance and the monotonicity of extermal length, gives
 \[\lambda_{\D} (0,F_{r_1}^*)\le\lambda_{\D} (0,F_{r_2}^*) \]
and hence $e^{-\pi \lambda_{\D}(0,F_r^*) }$ is decreasing in $r>0$. Similarly, $e^{-\pi \delta_{\D}(0,F_r^*) }$ is decreasing in $r>0$.
\end{proof}

\section{Proofs of the main results}\label{pro}

Next, we prove Theorems \ref{main} and \ref{hardynumber} and some consequent results.

\begin{proof}[Proof of Theorem \ref{main}] First, we prove the equivalence of (1) and (2). Suppose that $f\in A_{\alpha}^p (\D)$ for some $p>0$ and $\alpha \ge -1$. By Theorem \ref{known} we have
\begin{equation}\label{newre}
\int_{0}^{+\infty} r^{p-1} \omega_{\D} (0,F_r)^{\alpha+2} dr<+\infty.
\end{equation}

Now, let $\Omega_r$ be the component of $\D \backslash F_r$ containing $0$ and let $\Omega_r^*$ be the component of $\D \backslash F_r^*$ containing $0$. We set $F_r^{c}=\partial \Omega_r \backslash F_r $ and $F_r^{*c}=\partial \Omega_r^* \backslash F_r^*$. See Fig. \ref{disks}. Since $\Omega_r\subset\Omega_r^*$ and $F_r^{c}\subset F_r^{*c}$, it follows that, for every $r>0$,
\[ 1-\omega_{\Omega_r} (0,F_r)=\omega_{\Omega_r} (0,F_r^c)\le \omega_{\Omega_r^*} (0,F_r^{*c})=1-\omega_{\Omega_r^*} (0,F_r^*) \]
i.e.,
\begin{equation}\label{harmonic}
\omega_{\D} (0,F_r^*)\le \omega_{\D} (0,F_r).
\end{equation}
Since $\omega_{\D} (0,F_r^*)$ is measurable, by  Lemma \ref{decreasing}, relations (\ref{newre}) and (\ref{harmonic}) imply that
\[\int_{0}^{+\infty} r^{p-1} \omega_{\D} (0,F_r^*)^{\alpha+2} dr<+\infty.\]

Conversely, let $F_r'$ be a component of $F_r$ such that $d_{\D}(0,F_r)=d_{\D}(0,F_r')$. By the Beurling-Nevanlinna projection theorem \cite[p.\ 10]{Pog} we infer that, for every $r>0$,
\begin{equation}\label{projection}
e^{-d_{\D} (0,F_r)}=e^{-d_{\D} (0,F_r')}\le \frac{\pi}{2}\omega_{\D} (0,F_r') \le \frac{\pi}{2}\omega_{\D} (0,F_r^*).
\end{equation}
Therefore, if, for some $p>0$ and $\alpha \ge -1$,
\[\int_{0}^{+\infty} r^{p-1} \omega_{\D} (0,F_r^*)^{\alpha+2} dr<+\infty,\]
then
\[\int_{0}^{+\infty} r^{p-1}e^{-(\alpha+2)d_{\D} (0,F_r)} dr<+\infty\]
and hence Theorem \ref{known} implies that $f\in A_{\alpha}^p(\D)$.

The equivalence of (2), (3) and (4) is immediate by (\ref{haex}) and (\ref{hare}), which give, respectively, that
\begin{equation}\label{sx11}
\frac{\pi}{8}\omega_{\D}(0,F_r^*)\le e^{-\pi \lambda_{\D}(0,F_r^*)}\le \omega_{\D}(0,F_r^*)
\end{equation}
and
\begin{equation}\label{sx22}
\omega_{\D}(0,F_r^*)\le e^{-\pi \delta_{\D}(0,F_r^*)}\le \frac{\pi}{2} \omega_{\D}(0,F_r^*).
\end{equation}

\begin{figure}
\begin{overpic}[width=0.8\textwidth]{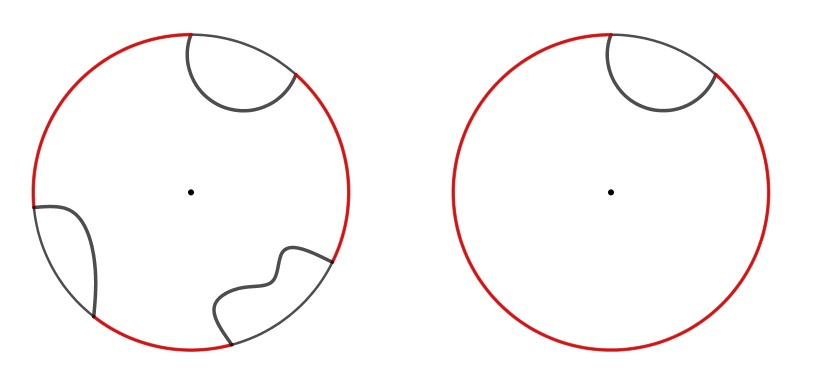}
\put  (15,10) {$F_r$}
\put  (70,35) {$F_r^*$}
\put  (22,20) {$0$}
\put  (74,20) {$0$}
\put  (1,35) {\color{red} $F_r^c$}
\put  (52,35) {\color{red}$F_r^{*c}$}
\end{overpic}
\caption{The sets $F_r,F_r^c,F_r^*$ and $F_r^{*c}$.}
\label{disks}
\end{figure}
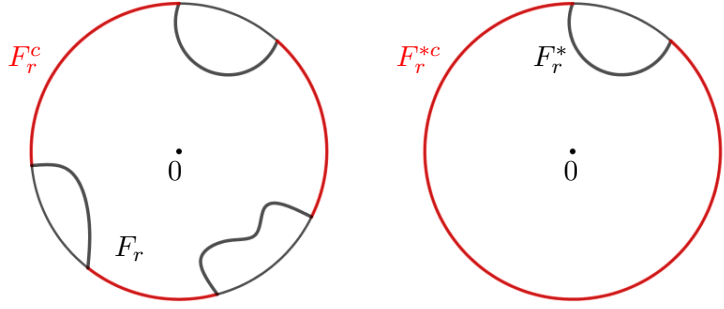
\end{proof}

\begin{corollary}\label{l2}
	Let $f$ be a conformal mapping of $\D$. Let $\Phi (r)$ denote $\omega_{\D}(0,F_r^*)$, $e^{-\pi \delta_{\D} (0,F_r^*)}$ or $e^{-\pi\lambda_{\D}(0,F_r^*)}$. If $f\in A_\alpha^p(\D)$ for some $p>0$ and $\alpha\ge-1$, then there is a  constant $C>0$ such that
	\[	\Phi (r)\le Cr^{-\frac{p}{\alpha+2}},\]
	for every $r>0$. Moreover, if there are $p'>0$, $\alpha '\ge-1$, $C>0$, and $r_0>0$ such that 
	\[\Phi (r)\le Cr^{-\frac{p'}{\alpha'+2}}\]
	for every $r>r_0$, then $f\in A_\alpha^p(\D)$ for all $p>0$ and $\alpha\geq -1$ such that $\frac{p}{\alpha+2}\in (0,\frac{p'}{\alpha'+2})$. 
\end{corollary}

\begin{proof}
If $f\in A_{\alpha}^p(\D)$ for some $p>0$ and $\alpha\ge -1$, by Theorem \ref{main} we have  
	\[\int_{0}^{+\infty}r^{p-1}\Phi (r)^{\alpha+2}dr<+\infty.
	\]
Lemma \ref{decreasing} implies that $\Phi (r)$ is decreasing in $r$ and thus, for $R>0$,
	\begin{align}
	\int_{0}^{+\infty}r^{p-1}\Phi (r)^{\alpha+2}dr \ge\int_{0}^{R}r^{p-1}\Phi (r)^{\alpha+2}dr \ge \frac{R^p}{p}\Phi (R)^{\alpha+2}. \nonumber
	\end{align}
	Combining the results above we infer that, for every $R>0$,
	\[\Phi (R)\le CR^{-\frac{p}{\alpha+2}},\]
	where 
	\[C=\left(p\int_{0}^{+\infty}r^{p-1}\Phi (r)^{\alpha+2}dr\right)^{\frac{1}{\alpha+2}}>0.\]
	
	Now, suppose there are $p'>0$, $\alpha '\geq -1$, $C>0$ and $r_0>0$ such that 
	\[\Phi (r)\le Cr^{-\frac{p'}{\alpha'+2}}\]
	for every $r>r_0$. If $\alpha\ge -1$ and $p>0$ satisfy $\frac{p}{\alpha+2}<\frac{p'}{\alpha'+2}$, then 
	\[\int_{r_0}^{+\infty}r^{p-1}\Phi (r)^{\alpha+2}dr\le C^{\alpha+2}\int_{r_0}^{+\infty}r^{p-1-\frac{p'}{\alpha'+2}(\alpha+2)}dr<+\infty.\]
So, by Theorem \ref{main} we deduce that $f\in A_\alpha^p(\D)$.
\end{proof}

\begin{proof}[Proof of Theorem \ref{hardynumber}] By (\ref{harmonic}) and (\ref{projection}) we have that, for every $r>0$,
\[ \frac{2}{\pi} e^{-d_{\D}(0,F_r)} \le \omega_{\D} (0,F_r^*)\le \omega_{\D} (0,F_r)\]
and thus
\[ \frac{\log \omega_{\D} (0,F_r)^{-1} }{\log r}\le \frac{\log \omega_{\D} (0,F_r^*)^{-1} }{\log r}\le \frac{\log (\pi/2)}{\log r}+ \frac{d_{\D} (0,F_r)}{\log r}.\]
This in combination with Theorem \ref{haknown} gives (1). Then equalities (2) and (3) follow directly by (\ref{sx11}) and (\ref{sx22}).
\end{proof}

\section{Counter-example}\label{exa}

The following example gives a negative answer to Question \ref{que}.

\begin{theorem}
There is a simply connected domain $D$ in $\C$ with the following property. Let $f$ be a conformal mapping from $\D$ onto $D$ with $f(0)=0$. Then $f\in A^p_{\alpha}(\D)$ for every $p>0$ and $\alpha \ge -1$ but 
\[\int_{0}^{+\infty} r^{p-1}e^{-\pi (\alpha+2) \delta_{\D} (0,F_r)} dr=+\infty\]
for every $p>0$ and $\alpha \ge -1$ that satisfy $p\ge \frac{\alpha }{2}+1$.
\end{theorem}

\begin{proof} If $\alpha_n={e^{4n\pi }}$ for every $n\ge 1$, let $D$ be the simply connected domain of Fig. \ref{t2}, namely,
\[D=\mathbb{C} \backslash \bigcup\limits_{k = 0}^3 {\left[ {{e^{i\frac{{k\pi }}{2}}}, + \infty } \right)} \backslash \bigcup\limits_{l = 1}^{ + \infty } {\bigcup\limits_{k = 0}^{{2^{l + 1}} - 1} {\left[ {{\alpha _l}{e^{i\frac{\pi }{{{2^l}}}\left( {\frac{1 }{2} + k} \right)}}, + \infty } \right)} }\]
with the notation $\left[ {r{e^{i\theta }}, + \infty } \right) = \left\{ {s{e^{i\theta }}:s \ge r} \right\}$. By the Riemann mapping theorem there is a conformal mapping $f$ from $\D$ onto $D$ such that $f(0)=0$. For $r>0$, let
\[\alpha_D(r)=\max \{\theta(E):E {\rm{\,\,is\,\, a\,\, subarc\,\, of\,\,}} D\cap\{z\in \C:|z|=r\}\},\]
where $\theta(E)$ denotes the angular Lebesgue measure of $E$. Since $D$ is a starlike domain with respect to $0$, the Hardy number of $D$ is given by
\[h(D)=\lim_{r\to +\infty}\frac{\pi}{\alpha_D(r)}=+\infty.\] 
See \cite[p.\ 237]{Han}. Thus, $f\in H^p(\D)$ for every $p>0$ and, since $H^p(\D) \subset A^p_{\alpha}(\D)$, we infer that $f\in A^p_{\alpha}(\D)$ for every $p>0$ and $\alpha>-1$. 

\begin{figure}
		\begin{center}
			\includegraphics[scale=0.6]{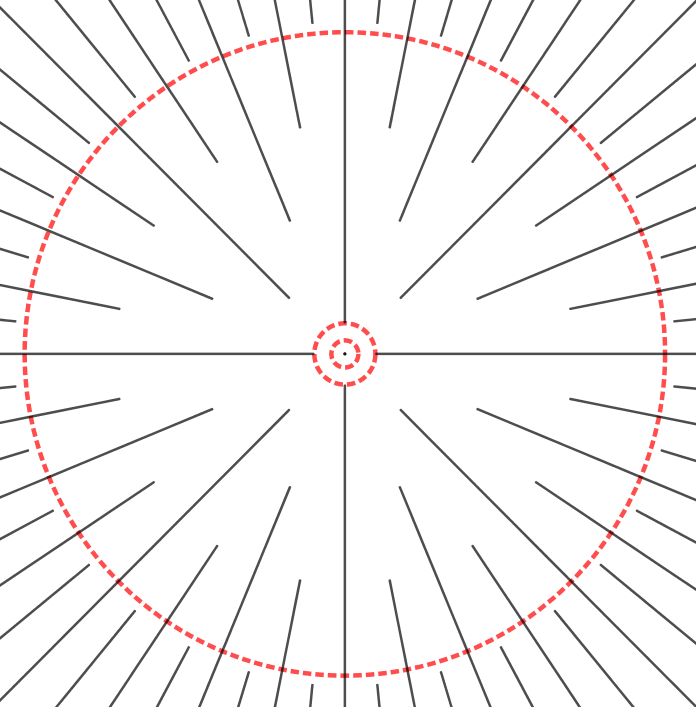}
			\caption{The domain $D$ and the circles $\partial B_{\varepsilon}, \partial \D$ and $\overline {f(F_r)}$.}
			\label{t2}
		\end{center}
\end{figure}

Now, for $\varepsilon \in (0,1)$, we set $B_{\varepsilon}=\{z\in \C: |z|<\varepsilon \}$. Also, let $D_r=D\cap \{z\in \C:|z|<r\}$. The conformal invariance of reduced extremal distance implies that, for $r>1$,
\begin{equation}\label{s1}
\delta_{\D}(0,F_r)=\delta_{D_r}(0,f(F_{r}))=\lim_{\varepsilon \to 0} (\tilde{\lambda}_{D_r\backslash B_{\varepsilon}}(\partial B_{\varepsilon}, f(F_r) )-\tilde{\lambda}_{D_r\backslash B_{\varepsilon}}(\partial B_{\varepsilon}, \partial D_r )).
\end{equation}
See Fig. \ref{t2}. We know that \cite[p.\ 142]{Gar}
\begin{equation}\label{s2}
\tilde{\lambda}_{D_r\backslash B_{\varepsilon}}(\partial B_{\varepsilon}, f(F_r) )=\frac{1}{2\pi}\log\frac{r}{\varepsilon}
\end{equation}
and if $\Delta=\{z\in \C: \varepsilon<|z|<1\}$, then kby the extension rule of extremal distance \cite[p.\ 134]{Gar} we have
\begin{equation}\label{s3}
\tilde{\lambda}_{D_r\backslash B_{\varepsilon}}(\partial B_{\varepsilon}, \partial D_r )\ge \tilde{\lambda}_{\Delta} (\partial B_{\varepsilon}, \partial \D) =\frac{1}{2\pi}\log\frac{1}{\varepsilon}.
\end{equation}
By (\ref{s1}), (\ref{s2}) and (\ref{s3}) we infer that, for every $r>1$,
\[\delta_{\D}(0,F_r)\le \frac{1}{2\pi} \lim_{\varepsilon \to 0}(\log\frac{r}{\varepsilon}-\log\frac{1}{\varepsilon})= \frac{1}{2\pi}\log r.\]
Therefore, it follows that, if $p\ge \frac{\alpha }{2}+1$, then
\begin{align}
\int_1^{+\infty}r^{p-1}e^{-\pi (\alpha+2)\delta_{\D}(0,F_{r})}dr \ge \int_1^{+\infty}r^{p-1-\frac{\alpha+2}{2}}dr=+\infty.  \nonumber
\end{align}
So, 
\[\int_{0}^{+\infty} r^{p-1}e^{-\pi (\alpha+2) \delta_{\D} (0,F_r)} dr=+\infty\]
for every $p>0$ and $\alpha \ge -1$ that satisfy $p\ge \frac{\alpha }{2}+1$.
\end{proof}

\end{document}